\documentclass[reqno,10pt]{amsart}
\usepackage{cite}

\usepackage{newtxtext}
\usepackage{mathptmx}
\DeclareMathAlphabet{\mathcal}{OMS}{cmsy}{m}{n}
\DeclareSymbolFont{largesymbols}{OMX}{cmex}{m}{n}

\usepackage[unicode, colorlinks]{hyperref}
\hypersetup{
	linkcolor = blue,
	citecolor = red,
	urlcolor = teal,
	bookmarksnumbered = true,
}

\usepackage{amsmath, amsthm, amssymb, bm, graphicx, mathrsfs}
\usepackage{cases}

\newtheorem{theorem}{Theorem}[section]
 
\newtheorem{proposition}[theorem]{Proposition}
\newtheorem{lemma}[theorem]{Lemma}

\newtheorem{corollary}[theorem]{Corollary}
\newtheorem{remark}[theorem]{Remark}

\newcommand{\e}{{\rm e}}
\newcommand{\ii}{{\rm i}}

\begin{document}

\title{Physical Space Proof of Bilinear Estimates and Applications to Nonlinear Dispersive Equations}

\author{Li Tu}
\address{School of Mathematical Sciences, Fudan University, Shanghai 200433, China.}
\email{23110180036@m.fudan.edu.cn}

\author{Yi Zhou}
\address{School of Mathematical Sciences, Fudan University, Shanghai 200433, China.}
\email{yizhou@fudan.edu.cn}

\date{\today}


\begin{abstract}
    We give a simpler proof for the local well-posedness of the modified Korteweg-de Vries equations and modified Benjamin-Ono equation in $H^{\frac{1}{4}}(\mathbb{R})$ and $H^{\frac{1}{2}}(\mathbb{R})$, respectively.
    The proof is based on the Strichartz estimate, dyadic decomposition and a bilinear estimate given by a new type of div-curl lemma.
\end{abstract}

\maketitle
\setcounter{tocdepth}{1}
\tableofcontents

\section{Introduction}

In this paper we study the initial value problem (IVP) for the (defocusing) modified Korteweg-de Vries (mKdV) equation\footnote{
    Our method can be also applied to the equation with nonlinearity $-\partial_x(u^3)$, i.e., the focusing case. 
}
\begin{equation}\label{mKdV}
    \begin{cases}
        \partial_tu + \partial_x^3u + \partial_x(u^3) = 0, \quad u\colon [0, T] \times \mathbb{R} \rightarrow \mathbb{R},  \\ 
        t = 0\colon u = \phi \in H^s(\mathbb{R}),
    \end{cases}
\end{equation}
and the modified Benjamin-Ono (mBO) equation 
\begin{equation}\label{mBO}
        \begin{cases}
            \partial_tu + \mathcal{H}\partial_x^2u + \partial_x(u^3) = 0, \quad u\colon [0, T] \times \mathbb{R} \rightarrow \mathbb{R},  \\ 
            t = 0\colon u = \phi \in H^s(\mathbb{R}), 
        \end{cases}
\end{equation}
where $\mathcal{H}$ is the Hilbert transform:
\begin{equation*}
    \mathcal{H}u(x) := \frac{1}{\pi}{\rm p.v.}\int_{\mathbb{R}}\frac{u(y)}{x - y} \,{\rm d}y.
\end{equation*}
Both of the two equations are mathematical models describing the weakly nonlinear long waves propagating in channels with dispersion effect (see \cite{SV}).

The IVP \eqref{mKdV} is invariant under the scaling transform 
\begin{equation*}
    u(t, x) \mapsto \lambda u(\lambda^3t, \lambda x)
\end{equation*}
and for \eqref{mBO}, 
\begin{equation*}
    u(x, t) \mapsto \lambda^{\frac{1}{2}}u(\lambda^2t, \lambda x),
\end{equation*}
which imply the critical index $s_c = -1/2$ and $s_c = 0$, respectively, in the sense that the homogeneous Sobolev norm $\dot H^{s_c}(\mathbb{R})$ is invariant under the scaling transform above.
It is known that the $L^2$-norm of the solution to \eqref{mKdV} and \eqref{mBO} is conserved:
\begin{equation*}
    \frac{{\rm d}}{{\rm d}t}\int_{\mathbb{R}}|u|^2 \,{\rm d}x = 0, \qquad {\rm when}\ u\ {\rm solves}\ \eqref{mKdV}\ {\rm or}\ \eqref{mBO}.
\end{equation*} 

For equation \eqref{mKdV} with quadratic nonlinearity
\begin{equation}\label{KdV}
    \partial_tu + \partial_x^3u + \frac{3}{2}\partial_x(u^2) = 0, \qquad u\colon [0, T] \times \mathbb{R} \rightarrow \mathbb{R},
\end{equation}
namely, the Korteweg-de Vries (KdV) equation, Miura \cite{M1} introduce a transform relating the solutions between \eqref{mKdV} and \eqref{KdV}:
\begin{equation}\label{Miura}
    \mathcal{M}u := \partial_xu + u^2
\end{equation}
Direct computation yields that if $u$ is a smooth (real-valued) solution to \eqref{mKdV}, then $\mathcal{M}u$ is a smooth (real-valued) solution to \eqref{KdV}.
Roughly speaking, the Miura transform \eqref{Miura} acts like a derivative, in particularly, mapping $H^s(\mathbb{R})$ into $H^{s - 1}(\mathbb{R})$ (one may refer to \cite[Section 9]{CCT} for more details),
and was used to construct an infinite number of conservation laws of the KdV equation \cite{MGK}.
An analogue of the Miura transform for Benjamin-Ono (BO) equation
\begin{equation}\label{BO}
    \partial_tu + \mathcal{H}\partial_x^2u + \frac{3}{2}\partial_x(u^2) = 0
\end{equation}
is the Bock-Kruskal transform \cite{BK}.
This (two-parameter) transform $u \mapsto v$ was defined implicitly via the formula 
\begin{equation*}
    2u = -\frac{1}{v + a}\mathcal{H}\partial_xw - \mathcal{H}\frac{\partial_xw}{v + a} + \frac{2av}{v + a},
\end{equation*}
where $a \in \mathbb{R}$ and $w$ is a real-valued function. 

There has been a large amount of work devoted to the low-regularity local well-posedness (LWP) theory in $H^s(\mathbb{R})$, i.e., trying to lower the best index $s$ for which one has the existence, uniqueness and continuous dependence on the initial data for a finite time interval, whose size depends on $\Vert \phi\Vert_{H^s}$, of the IVPs above. 
Here are some known results: 
For mKdV equation, it has been shown to be LWP for $s \geq 1/4$ by Kenig, Ponce and Vega \cite{KPV1} using local smoothing effects and maximal function estimate with some $L^q_tL^p_x$ estimates of Strichartz type.
It is worth mentioning that the method of multilinear estimate in Bourgain's spaces $X^{s, b}$ does not improve this result \cite{KPV2}, where the LWP of KdV equation for $s \geq -3/4$ was established.
For mBO equation, Kenig-Takaoka \cite{KT1} establish the LWP when $s \geq 1/2$ and Ionescu-Kenig \cite{IK} obtain the LWP for BO equation when $s \geq 0$.
Their proofs are all based on a refinement of Tao's gauge transformation \cite{Tao} and some $X^{s, b}$ techniques.

Note that the results mentioned above, except for BO equation\footnote{
    It was shown by Molinet, Saut and Tzvetkov \cite{MST} that the data-to-solution map of BO equation fails to be $C^2$ continuous from $H^s(\mathbb{R})$ to $H^s(\mathbb{R})$ for any $s \in \mathbb{R}$. 
    Moreover, this map is not uniformly continuous in any neighborhood of the origin for $s \geq 0$ \cite{KT2}.
}, are all sharp if one requires the uniform (or $C^k$) continuity of data-to-solution map; see \cite{KPV3, KT1, CCT}. 
For rougher initial data, the LWP is still avaliable when one requires weaker continuity; 
see the breakthrough of Harrop-Griffiths, Killip and Vi{\c{s}}an \cite{H-GKV1} for the LWP of complex mKdV equation in $H^s(\mathbb{R})$ when $s > -1/2$, and Killip-Laurens-Vi{\c{s}}an for BO equation when $s > -1/2$ \cite{KLV}, where they use the method of commuting flows introduced in \cite{KV1} that fully take advantage of the complete integrability of the equation.

The purpose of this paper is to give an alternative proof for LWP of IVPs \eqref{mKdV} and \eqref{mBO} when $s = 1/4$ and $s = 1/2$, respectively, using a div-curl lemma (see Section \ref{S4}) introduced by Zhou in the work of wave maps \cite{Z1}, and Wang-Zhou in their works of extremal hypersurface \cite{WZ1} and periodic Schr\"odinger flow \cite{WZ2};
see also their recent work of the bilinear estimate for wave equation \cite{WZ3}.

To introduce our idea, we start with the standard iteration scheme (taking mKdV equation \eqref{mKdV} as an example, and the operator $U(t)$ will be defined in Section \ref{S2}):
\begin{equation}\label{31A}
    \begin{cases}
        u^{(0)}(t) = U(t)\phi, \\
        \displaystyle u^{(j)}(t) = U(t)\phi - \int_0^tU(t - \tau)\partial_x((u^{(j - 1)})^3)(\tau) \,{\rm d}\tau, \quad j = 1, 2, \cdots.
    \end{cases}
\end{equation}
Equivalently, 
\begin{equation*}
    \begin{cases}
        \partial_tu^{(0)} + \partial_x^3u^{(0)} = 0, \\ 
        \partial_tu^{(j)} + \partial_x^3u^{(j)} + \partial_x((u^{(j - 1)})^3) = 0, \quad j = 1, 2, \cdots, \\ 
        u^{(j)}(0) = \phi, \quad j = 1, 2, \cdots.
    \end{cases}
\end{equation*}
To obtain the LWP, it suffices to prove that 
\begin{gather*}
    \Vert u^{(j)}\Vert_X \leq C, \\ 
    \Vert u^{(j + 2)} - u^{(j + 1)}\Vert_X \leq \frac{1}{2}\Vert u^{(j + 1)} - u^{(j)}\Vert_X
\end{gather*}
for all $j = 0, 1, 2, \cdots$, where $C > 0$ is a constant and $X$ is a function space to be determined.
It is known that the main difficulty is to find a suitable space to control the nonlinear term, or more accurately, to overcome the loss of derivatives from the nonlinearity during the iteration.
In this paper, we rely merely on Strichartz estimates to balance the integrability and regularity,
and by employing the div-curl lemma mentioned earlier, we can provide a simple \emph{physical space proof} for the bilinear estimates of the dyadic parts of the solution.
This argument allows us to transfer the derivatives from the high-frequency part to the low-frequency part. 
Consequently, we can control the nonlinearity by applying the (inhomogeneous) Littlewood-Paley decomposition to the solution.

As we have mentioned, one of the most crucial steps in our proof is the bilinear estimate proved by the div-curl lemma.
Precisely, let $u$ be the solution to the homogeneous equation of \eqref{mKdV} (resp., \eqref{mBO}).
Then for any dyadic numbers $\mu, \lambda$ satisfying $\mu \ll \lambda$, it holds that (the notations $u_{\mu}$ and $u_{\lambda}$ will be defined in Section \ref{S3})
\begin{equation}\label{D}
    \Vert u_{\mu}u_{\lambda}\Vert_{L_{t, x}^2} \lesssim \lambda^{-1}\Vert u_{\mu}(0)\Vert_{L^2}\Vert u_{\lambda}(0)\Vert_{L^2} \quad \left({\rm resp.,}\ \Vert u_{\mu}u_{\lambda}\Vert_{L_{t, x}^2} \lesssim \lambda^{-\frac{1}{2}}\Vert u_{\mu}(0)\Vert_{L^2}\Vert u_{\lambda}(0)\Vert_{L^2}\right).
\end{equation}
The bilinear estimate above for the homogeneous equation of \eqref{mKdV} (the free Airy equation) arose in \cite{Gr1} and it was proved by a \emph{phase space approach}.
In this aspect, our method seems to be simpler and more intuitive.

Now we state the main results of this paper (the Besov space $B^s_{p, q}$ will be defined in Section \ref{S3}):

\begin{theorem}\label{thm-mKdV}
    Suppose $s = 1/4$. Then there exists $T = T\left(\Vert \phi\Vert_{H^{\frac{1}{4}}(\mathbb{R})}\right) > 0$ and a unique solution $u = u(t)$ of the IVP \eqref{mKdV} in the time interval $[0, T]$ satisfying
    \begin{gather}
        \label{a1}u \in C\left([0, T]; H^{\frac{1}{4}}(\mathbb{R})\right) \cap L^8\left(0, T; B^{\frac{3}{8}}_{4, 2}(\mathbb{R})\right), \\ 
        \label{b1}\partial_x(u^3) \in L^{\frac{8}{7}}\left(0, T; B^{\frac{1}{8}}_{\frac{4}{3}, 2}(\mathbb{R})\right).
    \end{gather}
    Moreover, for any neighborhood $U$ of $\phi$ in $H^{\frac{1}{4}}(\mathbb{R})$ the data-to-solution map from $U$ into the class defined by \eqref{a1}-\eqref{b1} is Lipschitz continuous.
\end{theorem}

\begin{theorem}\label{thm-mBO}
    Suppose $s = 1/2$. Then there exists $T = T\left(\Vert \phi\Vert_{H^{\frac{1}{2}}(\mathbb{R})}\right) > 0$ and a unique solution $u = u(t)$ of the IVP \eqref{mBO} in the time interval $[0, T]$ satisfying
    \begin{gather}
        \label{c1}u \in C\left([0, T]; H^{\frac{1}{2}}(\mathbb{R})\right) \cap L^8\left(0, T; B^{\frac{1}{2}}_{4, 2}(\mathbb{R})\right), \\ 
        \label{d1}\partial_x(u^3) \in L^{\frac{8}{7}}\left(0, T; B^{\frac{1}{2}}_{\frac{4}{3}, 2}(\mathbb{R})\right).
    \end{gather}
    Moreover, for any neighborhood $U$ of $\phi$ in $H^{\frac{1}{2}}(\mathbb{R})$ the data-to-solution map from $U$ into the class defined by \eqref{c1}-\eqref{d1} is Lipschitz continuous.
\end{theorem}

It is notably that the method we present here can be also applied to the following dispersion-generalized Benjamin-Ono equation:
\begin{equation}\label{mgBO}
    \begin{cases}
        \partial_tu + D_x^{1 + \alpha}\partial_xu + \partial_x(u^3) = 0, \quad u\colon [0, T] \times \mathbb{R} \rightarrow \mathbb{R}, \\
        t = 0\colon u = \phi \in H^s(\mathbb{R}),
    \end{cases}
\end{equation}
where $D_x$ is the Fourier multiplier with symbol $|\xi|$ and $0 < \alpha < 1$. 
When $\alpha = 1$, equation \eqref{mgBO} is the mKdV equation \eqref{mKdV} and when $\alpha = 0$, it reduces to the mBO equation \eqref{mBO}. 
It has been proved that equation \eqref{mgBO} is LWP when $s \geq 1/2 - \alpha/4$ by Guo \cite{G1} using some $X^{s, b}$ estimates.
Using the div-curl lemma and we can prove that the following bilinear estimate holds for the solution to the homogeneous equation of \eqref{mgBO}:
\begin{equation}\label{be-mgBO}
    \Vert u_{\mu}u_{\lambda}\Vert_{L_{t, x}^2} \lesssim \lambda^{-\frac{1 + \alpha}{2}}\Vert u_{\mu}(0)\Vert_{L^2}\Vert u_{\lambda}(0)\Vert_{L^2} \qquad (\mu \ll \lambda),
\end{equation}
where $\mu$ and $\lambda$ are dyadic numbers. Then it can be also shown from Strichartz estimate, dyadic decomposition and bilinear estimate \eqref{be-mgBO} that IVP \eqref{mgBO} is LWP in $H^{\frac{1}{2} - \frac{\alpha}{4}}(\mathbb{R})$ with regularity
\begin{equation}\label{abc}
    u \in L^8\left(0, T; B^{\frac{1}{2} - \frac{\alpha}{8}}_{4, 2}(\mathbb{R})\right), \quad \partial_x(u^3) \in L^{\frac{8}{7}}\left(0, T; B^{\frac{1}{2} - \frac{3\alpha}{8}}_{\frac{4}{3}, 2}(\mathbb{R})\right).
\end{equation}

We end this section with some reductions. Indeed, besides the bilinear estimates \eqref{D} proved by a physical space approach, the rest steps of the proof for Theorem \ref{thm-mKdV} and \ref{thm-mBO} is standard. 

Without loss of generality, it suffices to consider the proof of Theorem \ref{thm-mKdV}, and Theorem \ref{thm-mBO} follows similarly after the bilinear estimate is proved (see Theorem \ref{be-BO}).
It is clear that Theorem \ref{thm-mKdV} follows immediately from the proposition below:

\begin{proposition}\label{prop4.1}
    Let $\{u^{(j)}\}$ be the iteration sequence defined by \eqref{31A}. Then for all $j = 0, 1, 2, \cdots$, there exists $T = T\left(\Vert \phi\Vert_{H^{\frac{1}{4}}(\mathbb{R})}\right) > 0$ and constant $C > 0$ such that 
    \begin{gather}
        \label{37}\Vert u^{(j)}\Vert_{S(T)} \leq 2C\Vert \phi\Vert_{H^{\frac{1}{4}}(\mathbb{R})}, \\ 
        \label{38}\Vert \partial_x((u^{(j)})^3)\Vert_{N(T)} \leq \Vert \phi\Vert_{H^{\frac{1}{4}}(\mathbb{R})}, \\
        \label{39}\Vert u^{(j + 2)} - u^{(j + 1)}\Vert_{S(T)} \leq \frac{1}{2}\Vert u^{(j + 1)} - u^{(j)}\Vert_{S(T)}, \\ 
        \label{40A}\Vert \partial_x((u^{(j + 2)})^3 - (u^{(j + 1)})^3)\Vert_{N(T)} \leq \frac{1}{2}\Vert \partial_x((u^{(j + 1)})^3 - (u^{(j)})^3)\Vert_{N(T)}.
    \end{gather}
    where 
    \begin{gather*}
        \Vert \cdot \Vert_{S(T)} := \Vert \cdot \Vert_{L^{\infty}\left(0, T; H^{\frac{1}{4}}(\mathbb{R})\right)} + \Vert \cdot \Vert_{L^8\left(0, T; B^{\frac{3}{8}}_{4, 2}(\mathbb{R})\right)}, \\ 
        \Vert \cdot \Vert_{N(T)} := \Vert \cdot \Vert_{L^{\frac{8}{7}}\left(0, T; B^{\frac{1}{8}}_{\frac{4}{3}, 2}(\mathbb{R})\right)}.
    \end{gather*}
    Moreover, the bilinear estimate 
    \begin{equation}\label{BE1}
        \Vert u_{\mu}^{(j)}u_{\lambda}^{(j)}\Vert_{L^2(0, T; L^2(\mathbb{R}))} \lesssim \lambda^{-1}\Vert u_{\mu}^{(j)}(0)\Vert_{L^2(\mathbb{R})}\Vert u_{\lambda}^{(j)}(0)\Vert_{L^2(\mathbb{R})}.
    \end{equation}
    holds for all $j = 0, 1, 2, \cdots$, where $\mu$ and $\lambda$ are dyadic numbers satisfying $\mu \ll \lambda$.
\end{proposition}

\subsection{Outline of the paper}

The remaining part of this paper is organized as follows: In Section \ref{S2} we will introduce some notations.
In Section \ref{S3}, we will introduce the function spaces, including Littlewood-Paley theory, Besov spaces and Strichartz estimate.
In Section \ref{S4} we use the div-curl lemma to give a physical space proof of the bilinear estimates \eqref{D}. 
Finally, in Section \ref{S5} we will complete the proof of Proposition \ref{prop4.1} using dyadic decomposition and the bilinear estimates proved in Section \ref{S3}. 

\section{Notations}\label{S2}

We write $X \lesssim Y$ to indicate that $X \leq CY$, and $X \gtrsim Y$ to indicate that $X \geq CY$ for some positive constant $C$.
The symbol $X \sim Y$ implies that $X \lesssim Y$ and $X \gtrsim Y$.
Denote by $\mathbb{N} := \{0, 1, 2, \cdots\}$ the natural numbers set.
The Greek letters $\lambda$, $\mu$ and $\sigma$ are used to denote dyadic numbers in this paper.

We use $L^p, \ell^p$ to denote the usual Lebesgue spaces, and $W^{k, p}, H^k := W^{k, 2}$ to denote the usual Sobolev spaces.
The spacetime norm $\Vert \cdot\Vert_{L_t^qL_x^p}$ is defined by 
\begin{equation*}
    \Vert u\Vert_{L_t^qL_x^p} = \Vert u\Vert_{L^q(0, T; L^p(\mathbb{R}))} := \left(\int_0^T\left(\int_{\mathbb{R}}|u|^p \,{\rm d}x\right)^{\frac{q}{p}} \,{\rm d}t\right)^{\frac{1}{q}}.
\end{equation*}
When $p = q$, we will also abbreviate $L_t^pL_x^p$ as $L_{t, x}^p$.

Given a function $u$, the (spatial) Fourier transform of $u$ is defined by 
\begin{equation*}
    \widehat{u}(\xi) = \mathscr{F}_xu(\xi) := \int_{\mathbb{R}}\e^{-\mathrm{i}x\xi}u(x) \,{\rm d}x,
\end{equation*}
and the symbol $\mathscr{F}_{\xi}^{-1}$ denotes the Fourier inversion with respect to the variable $\xi$:
\begin{equation*}
    \mathscr{F}_{\xi}^{-1}u(\xi) := \frac{1}{2\pi}\int_{\mathbb{R}}\e^{\ii x\xi}u(\xi) \,{\rm d}\xi.
\end{equation*}
Let $D_x^s$ be the Fourier multiplier with symbol $|\xi|^s$. Since $\widehat{\mathcal{H}u}(\xi) = -\ii\ {\rm sgn}(\xi)\widehat{u}(\xi)$, where $\mathcal{H}$ denotes the Hilbert transform, it follows that $D_x = \partial_x \circ \mathcal{H} = \mathcal{H} \circ \partial_x$.

For the free Airy equation
\begin{equation}\label{Airy}
    \partial_tu + \partial_x^3u = 0,
\end{equation}
we denote by 
\begin{equation*}
    u(t) := U(t)\phi
\end{equation*}
the solution of \eqref{Airy} with initial data $\phi$, where
\begin{equation*}
    U(t)\colon \phi \mapsto \frac{1}{2\pi}\int_{\mathbb{R}}\int_{\mathbb{R}}\e^{\ii x\xi}\e^{\ii t\xi^3}\widehat{\phi}(\xi) \,{\rm d}x{\rm d}\xi.
\end{equation*}
It follows from Duhamel's principle that the solution of 
\begin{equation*}
    \begin{cases}
        \partial_tu + \partial_x^3u + F = 0, \\ 
        t = 0\colon u = \phi
    \end{cases}
\end{equation*}
is given by 
\begin{equation*}
    u(t) = U(t)\phi - \int_0^tU(t - \tau)F(\tau) \,{\rm d}\tau.
\end{equation*}

\section{Function Spaces}\label{S3}

In this section, we will introduce the Besov spaces and the Strichartz estimates on it.

\subsection{Littlewood-Paley theory and the Besov spaces}\label{S3.1}

Let $\psi$ be a smooth non-neagtive radial function on $\mathbb{R}$ such that ${\rm supp}\ \psi \subseteq \{8/9 \leq |\xi| \leq 9/8\}$.
Define $\varphi(\xi) = \psi(\xi/2) - \psi(\xi)$, then the support of $\varphi$ is contained in $ \{8/9 \leq |\xi| \leq 9/4\}$, and we have  
\begin{equation*}
    \psi(\xi) + \sum_{\lambda \geq 1}\varphi(\lambda^{-1}\xi) = 1, \qquad \forall \xi.
\end{equation*}
For dyadic numbers $\mu$ and $\lambda$, we define the operators $P_{\lambda}$ and $S_{\lambda}$ by
\begin{equation*}
    P_{\lambda}\colon u \mapsto \mathscr{F}_{\xi}^{-1}(\varphi(\lambda^{-1}\cdot)\mathscr{F}_xu) =: u_{\lambda}, \quad S_{\lambda} = \sum_{\mu < \lambda}P_{\mu}.
\end{equation*}
Denote by $u_{\lambda} := P_{\lambda}u$. It follows from the definition above that ${\rm supp}\ \widehat{u_{\lambda}} \subseteq \{\xi\colon |\xi| \sim \lambda\}$.
We will frequently use the following \emph{Bernstein's inequality} in this paper: For $s \in \mathbb{R}$ and $1 \leq p \leq \infty$, it holds that 
\begin{equation*}
    \Vert D_x^su_{\lambda}\Vert_{L_x^p} \sim \lambda^s\Vert u_{\lambda}\Vert_{L_x^p}, \qquad \forall \lambda\ {\rm dyadic}.
\end{equation*}

The (inhomogeneous) dyadic blocks are defined by 
\begin{equation*}
    \Delta_{\lambda}u := 
    \begin{cases}
        u_{\lambda} \quad &\lambda \geq 1, \\ 
        S_1u \quad &\lambda = 1/2, \\ 
        0 \quad &\lambda < 1/2.
    \end{cases}
\end{equation*}
Then we have the (inhomogeneous) Littlewood-Paley decomposition (that we will use throughout this paper):
\begin{equation*}
    u = \sum\Delta_{\lambda}u.
\end{equation*}

For $s \in \mathbb{R}, 1 \leq p, q \leq \infty$, the (inhomogeneous) Besov norm is defined by
\begin{align*}
    \Vert u\Vert_{B^s_{p, q}(\mathbb{R})} &:= \Vert \{\lambda^s\Vert \Delta_{\lambda}u\Vert_{L^p(\mathbb{R})}\}\Vert_{\ell^q} \\
    &= \Vert S_1u\Vert_{L^p(\mathbb{R})} + 
    \begin{cases}
        \displaystyle \left(\sum_{\lambda \geq 1}(\lambda^s\Vert u_{\lambda}\Vert_{L^p(\mathbb{R})})^q\right)^{1/q} \quad &q < \infty, \\ 
        \sup_{\lambda \geq 1}\lambda^s\Vert u_{\lambda}\Vert_{L^p(\mathbb{R})} \quad &q = \infty.
    \end{cases}
\end{align*}
The Besov spaces $B^s_{p, q}(\mathbb{R})$ is then defined as the completion of Schwartz functions with respect to the norm $\Vert \cdot\Vert_{B^s_{p, q}(\mathbb{R})}$.
Here are some basic properties of Besov spaces (one may refer to \cite{BCD1} for the proof):
\begin{itemize}
    \item The following inclusions hold:
    \begin{equation}\label{ab}
        B^{s_1}_{p, q} \hookrightarrow B^{s_2}_{p, q}\ {\rm when} \ s_1 \geq s_2, \qquad B^s_{p, q_1} \hookrightarrow B^s_{p, q_2}\ {\rm when}\ q_1 \leq q_2;
    \end{equation}
    \item For any $s \in \mathbb{R}$, we have $H^s = B^s_{2, 2}$.
\end{itemize}

\subsection{Strichartz estimates}

Next we recall the Strichartz estimate (see \cite{KPV4, MT} and the references therein).

\begin{theorem}[Strichartz estimate]\label{ST}
    Suppose $0 \leq \alpha \leq 1$. Let $u$ be a solution of the following IVP:
    \begin{equation*}
        \begin{cases}
            \partial_tu + D^{1 + \alpha}\partial_xu + F = 0, \quad u\colon [0, T] \times \mathbb{R} \rightarrow \mathbb{R}, \\ 
            t = 0\colon u = \phi.
        \end{cases}
    \end{equation*}
    Then for all $0 \leq \theta \leq 1$ we have 
    \begin{equation*}
        \Vert u \Vert_{L^{\infty}_tH^s_x} + \Vert u \Vert_{L^{4/\theta}_tB^r_{2/(1 - \theta), 2}} \lesssim \Vert \phi\Vert_{H^s_x} + \Vert F\Vert_{L^{(4/\theta)'}_tB^{r^*}_{(2/(1 - \theta))', 2}},
    \end{equation*}
    where $r \in \mathbb{R}, s = r - \alpha\theta/4, r^* = r - \alpha\theta/2$, and $p'$ of number is conjugate of $p \in [1, \infty]$ given by $1/p + 1/p' = 1$. 
\end{theorem}

To determine the integrability and regularity indices ($\theta$ and $r$) in our application, we employ a herusitic argument as follows:
On the one hand, balancing the integrability with respect to variable $x$ between $u$ and the nonlinearity $\partial_x(u^3)$ leads to 
\begin{equation*}
    \frac{2}{1 - \theta} = 3\left(\frac{2}{1 - \theta}\right)' \Longrightarrow \theta = \frac{1}{2}.
\end{equation*} 
Note that the integrability for time variable $t$ can then be balanced by H\"older's inequality since when $\theta = 1/2$, we have 
\begin{equation*}
    8 = \frac{4}{\frac{1}{2}} > 3 \cdot \left(\frac{4}{\frac{1}{2}}\right)' = \frac{24}{7}.
\end{equation*}
On the other hand, to spread the derivatives in the nonlinearity $\partial_x(u^3)$ and then balance the regularity, it follows that (take $\theta = 1/2$)
\begin{equation*}
    r = \frac{r^* + 1}{3} = \frac{1}{3}\left(r - \frac{\alpha}{4} + 1\right) \Longrightarrow r = \frac{1}{2} - \frac{\alpha}{8}.
\end{equation*}
To conclude, the Strichartz estimate we will use in this paper is given by:
\begin{equation*}
    \Vert u \Vert_{L^{\infty}_tH^{\frac{1}{2} - \frac{\alpha}{4}}_x} + \Vert u \Vert_{L^8_tB^{\frac{1}{2} - \frac{\alpha}{8}}_{4, 2}} \lesssim \Vert \phi\Vert_{H^{\frac{1}{2} - \frac{\alpha}{4}}_x} + \Vert F\Vert_{L^{\frac{8}{7}}_tB^{\frac{1}{2} - \frac{3\alpha}{8}}_{\frac{4}{3}, 2}},
\end{equation*}
which also imply the regularity of the solution as stated in \eqref{abc}.

Particularly, in the case $\alpha = 1$ we have  
\begin{equation}\label{SE-mKdV}
    \Vert u\Vert_{L^{\infty}_tH^{\frac{1}{4}}_x} + \Vert u\Vert_{L^8_tB^{\frac{3}{8}}_{4, 2}} \lesssim \Vert \phi\Vert_{H^{\frac{1}{4}}_x} + \Vert F\Vert_{L^{\frac{8}{7}}_tB^{\frac{1}{8}}_{\frac{4}{3}, 2}};
\end{equation}
and for $\alpha = 0$, it follows that 
\begin{equation}\label{SE-mBO}
    \Vert u\Vert_{L^{\infty}_tH^{\frac{1}{2}}_x} + \Vert u\Vert_{L^8_tB^{\frac{1}{2}}_{4, 2}} \lesssim \Vert \phi\Vert_{H^{\frac{1}{2}}_x} + \Vert F\Vert_{L^{\frac{8}{7}}_tB^{\frac{1}{2}}_{\frac{4}{3}, 2}}.
\end{equation} 

\section{Div-Curl Lemma and Bilinear Estimates}\label{S4}

In this section, we give a physical space proof for the bilinear estimates \eqref{D}. 
The main tool is the following div-curl lemma. 

\begin{lemma}[Div-curl lemma]
    Suppose that
\begin{equation*}
    \begin{cases}
        \partial_tf^{11} + \partial_xf^{12} = G^1, \\ 
        \partial_tf^{21} - \partial_xf^{22} = G^2,
    \end{cases}
    \qquad (t, x) \in [0, T] \times \mathbb{R},
\end{equation*}
and $f^{ij} \rightarrow 0, x \rightarrow +\infty, i, j = 1, 2$. 
Then we have 
\begin{equation}\label{d-c}
\begin{aligned}
    \int_0^T\int_{\mathbb{R}}f^{11}f^{22} + f^{12}f^{21} \lesssim \left(\Vert f^{11}(0)\Vert_{L^1_x} + \Vert f^{11}\Vert_{L^{\infty}_tL^1_x} + \Vert G^1\Vert_{L^1_{t, x}}\right)\left(\Vert f^{21}(0)\Vert_{L^1_x} + \Vert f^{21}\Vert_{L^{\infty}_tL^1_x} + \Vert G^2\Vert_{L^1_{t, x}}\right),
\end{aligned}
\end{equation}
provided that the right-side is bounded.
\begin{proof}
    See \cite{WZ3}.
\end{proof}
\end{lemma}

We now turn to the proof of bilinear estimates \eqref{D}. Firstly, we consider the case of mKdV equation \eqref{mKdV}.

\begin{proposition}[Derivatives transferred from high-frequency to low-frequency]\label{prop}
    Let $u$ satisfy the free Airy equation $\partial_tu + \partial_x^3u = 0$.
    Then for any dyadic numbers $\lambda$ and $\mu$, it holds that
    \begin{equation*}
        \Vert u_{\mu}\partial_xu_{\lambda}\Vert_{L_{t, x}^2}^2 - \Vert u_{\lambda}\partial_xu_{\mu}\Vert_{L_{t, x}^2}^2 \lesssim \Vert u_{\mu}(0)\Vert_{L_x^2}^2\Vert u_{\lambda}(0)\Vert_{L_x^2}^2.
    \end{equation*}
    \begin{proof}
        Since $P_{\mu}$ and $P_{\lambda}$ commute with $\partial$, it follows that $u_{\mu}$ and $u_{\lambda}$ satisfy the following equations:
        \begin{equation}\label{6A}
            \begin{cases}
               \partial_t \displaystyle\left(\frac{u_{\mu}^2}{2}\right) + \partial_x\left(u_{\mu}\partial_x^2u_{\mu} - \frac{(\partial_xu_{\mu})^2}{2}\right) = 0, \\ 
               \partial_t \displaystyle\left(\frac{u_{\lambda}^2}{2}\right) + \partial_x\left(u_{\lambda}\partial_x^2u_{\lambda} - \frac{(\partial_xu_{\lambda})^2}{2}\right) = 0.
            \end{cases}
        \end{equation}
        Let 
        \begin{equation*}
            \begin{cases}
                \displaystyle f^{11} = \frac{u_{\mu}^2}{2}, \\ 
                \displaystyle f^{12} = u_{\mu}\partial_x^2u_{\mu} - \frac{(\partial_xu_{\mu})^2}{2}, \\ 
                \displaystyle f^{21} = \frac{u_{\lambda}^2}{2}, \\ 
                -\displaystyle f^{22} = u_{\lambda}\partial_x^2u_{\lambda} - \frac{(\partial_xu_{\lambda})^2}{2}.
            \end{cases}
        \end{equation*}
        On the one hand, direct computation shows that 
        \begin{align*}
            &f^{11}f^{22} + f^{12}f^{21} \\ 
            = &-\frac{u_{\mu}^2}{2}\left(u_{\lambda}\partial_x^2u_{\lambda} - \frac{(\partial_xu_{\lambda})^2}{2}\right) + \frac{u_{\lambda}^2}{2}\left(u_{\mu}\partial_x^2u_{\mu} - \frac{(\partial_xu_{\mu})^2}{2}\right) \\ 
            = &\frac{1}{4}(u_{\mu}^2(\partial_xu_{\lambda})^2 - u_{\lambda}^2(\partial_xu_{\mu})^2) - \frac{1}{2}(u_{\mu}^2u_{\lambda}\partial_x^2u_{\lambda} - u_{\lambda}^2u_{\mu}\partial_x^2u_{\mu}) \\
            = &\frac{3}{4}(u_{\mu}^2(\partial_xu_{\lambda})^2 - u_{\lambda}^2(\partial_xu_{\mu})^2) - \frac{1}{2}(u_{\mu}^2\partial_x(u_{\lambda}\partial_xu_{\lambda}) - u_{\lambda}^2\partial_x(u_{\mu}\partial_xu_{\mu})).
        \end{align*}
        Integrating by parts and we obtain  
        \begin{equation}\label{6}
        \begin{aligned}
            &\int_0^T\int_{\mathbb{R}}f^{11}f^{22} + f^{12}f^{21} \\ 
            = &\frac{3}{4}\left(\Vert u_{\mu}\partial_xu_{\lambda}\Vert_{L_{t, x}^2}^2 - \Vert u_{\lambda}\partial_xu_{\mu}\Vert_{L_{t, x}^2}^2\right) + \frac{1}{2}\int_0^T\int_{\mathbb{R}}u_{\lambda}\partial_xu_{\lambda}\partial_x(u_{\mu}^2) - u_{\mu}\partial_xu_{\mu}\partial_x(u_{\lambda}^2) \\ 
            = &\frac{3}{4}\left(\Vert u_{\mu}\partial_xu_{\lambda}\Vert_{L_{t, x}^2}^2 - \Vert u_{\lambda}\partial_xu_{\mu}\Vert_{L_{t, x}^2}^2\right).
        \end{aligned}
    \end{equation}
        On the other hand, we have 
        \begin{equation}\label{7}
            \Vert f^{11}(0)\Vert_{L^1_x} = \frac{1}{2}\Vert u_{\mu}(0)\Vert_{L^2_x}^2, \quad \Vert f^{21}(0)\Vert_{L_x^1} = \frac{1}{2}\Vert u_{\lambda}(0)\Vert_{L_x^2}^2,
        \end{equation}
        and the conservation laws show that 
        \begin{equation}\label{8}
            \Vert f^{11}\Vert_{L^{\infty}_tL^1_x} \lesssim \Vert u_{\mu}(0)\Vert_{L^2_x}^2, \quad \Vert f^{21}\Vert_{L_t^{\infty}L^1_x} \lesssim \Vert u_{\lambda}(0)\Vert_{L^2_x}^2.
        \end{equation}
        Combining \eqref{6}-\eqref{8} with the div-curl lemma \eqref{d-c} and we complete the proof of the proposition.
    \end{proof}
\end{proposition}

To prove the bilinear estimate we need the following lemma.
For a function $u$ on $\mathbb{R}$ and $y \in \mathbb{R}$, we define the \emph{translation} of $u$ by 
\begin{equation*}
    u^y(x) := u(x - y).
\end{equation*}

\begin{lemma}\label{lma}
    For dyadic number $\lambda \geq 1$, there holds
    \begin{gather}
        \label{a}\Vert u_{\mu}u_{\lambda}\Vert_{L_{t, x}^2} \lesssim \lambda^{-1}\sup_y\Vert u_{\mu}D_xu_{\lambda}^y\Vert_{L_{t, x}^2}, \\ 
        \label{b}\Vert u_{\mu}D_xu_{\lambda}\Vert_{L_{t, x}^2} \lesssim \lambda\sup_y\Vert u_{\mu}u_{\lambda}^y\Vert_{L_{t, x}^2}.
    \end{gather}
    \begin{proof}
        It suffices to prove \eqref{a}, and \eqref{b} can be proved in a similar way.
        Let $\chi$ be the compactly supported smooth function satisfying ${\rm supp}\ \varphi \subseteq {\rm supp}\ \chi$, where $\varphi$ is defined in Section \ref{S3.1}.
        It follows that 
        \begin{equation*}
            \widehat{u_{\lambda}} = \frac{1}{|\xi|}\widehat{D_xu_{\lambda}} = \lambda^{-1}\left(\chi\left(\frac{|\xi|}{\lambda}\right)\frac{\lambda}{|\xi|}\right)\widehat{D_xu_{\lambda}} =: \lambda^{-1}\widehat{h}\widehat{D_xu_{\lambda}}.
        \end{equation*}
        Therefore, $u_{\lambda} \sim \lambda^{-1}h \ast D_xu_{\lambda}$.
        Note that $\Vert h\Vert_{L^1} = C < +\infty$, where the constant $C$ is independent of $\lambda$, then it follows from Minkowski's inequality that 
        \begin{align*}
            \Vert u_{\mu}u_{\lambda}\Vert_{L^2_{x, t}} \sim \lambda^{-1}\Vert u_{\mu}(h \ast D_xu_{\lambda})\Vert_{L^2_{x, t}} &\leq \lambda^{-1} \int_{\mathbb{R}}|h(y)|\Vert u_{\mu}D_xu_{\lambda}^y\Vert_{L_{t, x}^2} \,{\rm d}y \\
            &\lesssim \lambda^{-1}\sup_y\Vert u_{\mu}D_xu_{\lambda}^y\Vert_{L_{t, x}^2}.
        \end{align*}
    \end{proof}
\end{lemma}

It is clear that \eqref{a} and \eqref{b} are also valid when $D_x$ is replaced with $\partial_x$.

\begin{corollary}[Bilinear estimate I]\label{be-KdV}
    Let $u$ satisfy the free Airy equation $\partial_tu + \partial_x^3u = 0$.
    Then for any dyadic numbers $\lambda$ and $\mu$ satisfying $\mu \ll \lambda$, it holds that
    \begin{equation}\label{BE-KdV}
        \Vert u_{\mu}u_{\lambda}\Vert_{L_{t, x}^2} \lesssim \lambda^{-1}\Vert u_{\mu}(0)\Vert_{L_x^2}\Vert u_{\lambda}(0)\Vert_{L_x^2}.
    \end{equation}
    \begin{proof}
        Note that for all $a, b \in \mathbb{R}$, $u_{\mu}^a, u_{\lambda}^b$ also satisfy \eqref{6A}. Therefore, we can replace $u_{\mu}$ and $u_{\lambda}$ in the proof of Proposition \ref{prop} with $u_{\mu}^a$ and $u_{\lambda}^b$, respectively, while keeping the rest unchanged.
        At this point, \eqref{6A} becomes
        \begin{equation}\label{13}
            \Vert u_{\mu}^a\partial_xu_{\lambda}^b\Vert_{L_{t, x}^2}^2 - \Vert u_{\lambda}^b\partial_xu_{\mu}^a\Vert_{L_{t, x}^2}^2 \lesssim \Vert u_{\mu}^a(0)\Vert_{L_x^2}^2\Vert u_{\lambda}^b(0)\Vert_{L_x^2}^2 = \Vert u_{\mu}(0)\Vert_{L_x^2}^2\Vert u_{\lambda}(0)\Vert_{L_x^2}^2.
        \end{equation}
        Applying Lemma \ref{lma} on $u_{\mu}^a, u_{\lambda}^b$ and we obtain 
        \begin{gather}\label{c}
            \Vert u_{\mu}^au_{\lambda}^b\Vert_{L_{t, x}^2}^2 \lesssim \lambda^{-2}\sup_y\Vert u_{\mu}^a\partial_xu_{\lambda}^{b + y}\Vert_{L_{t, x}^2}^2 = \lambda^{-2}\sup_y\Vert u_{\mu}^a\partial_xu_{\lambda}^y\Vert_{L_{t, x}^2}^2.
        \end{gather}
        Taking supremum over $a$ and $b$ on both sides of \eqref{c} and we have 
        \begin{equation*}
            \sup_{a, b}\Vert u_{\mu}^au_{\lambda}^b\Vert_{L_{t, x}^2}^2 \lesssim \lambda^{-2}\sup_{a, b}\Vert u_{\mu}^a\partial_xu_{\lambda}^b\Vert_{L_{t, x}^2}^2,
        \end{equation*}
        which is equivalent to 
        \begin{equation}\label{15}
            \sup_a\Vert u_{\mu}^au_{\lambda}\Vert_{L_{t, x}^2}^2 \lesssim \lambda^{-2}\sup_a\Vert u_{\mu}^a\partial_xu_{\lambda}\Vert_{L_{t, x}^2}^2.
        \end{equation}
        Similarly, it holds that 
        \begin{equation}\label{16}
            \sup_a\Vert u_{\lambda}\partial_xu_{\mu}^a\Vert_{L_{t, x}^2}^2 \lesssim \mu^2\sup_a\Vert u_{\lambda}u_{\mu}^a\Vert_{L_{t, x}^2}^2.
        \end{equation}
        Now taking supremum over $a$ on both sides of \eqref{13} (with $b = 0$), it follows from \eqref{15} and \eqref{16} that 
        \begin{equation}\label{17}
        \begin{aligned}
            \Vert u_{\mu}(0)\Vert_{L_x^2}^2\Vert u_{\lambda}(0)\Vert_{L_x^2}^2 \gtrsim \sup_a\Vert u_{\mu}^a\partial_xu_{\lambda}\Vert_{L_{t, x}^2}^2 - \sup_a\Vert u_{\lambda}\partial_xu_{\mu}^a\Vert_{L_{t, x}^2}^2 \gtrsim (\lambda^2 - \mu^2)\sup_a\Vert u_{\mu}^au_{\lambda}\Vert_{L_{t, x}^2}^2.
        \end{aligned}
    \end{equation}
        Finally, since $\mu \ll \lambda$, we have that 
        \begin{equation}\label{10}
            (\lambda^2 - \mu^2)\sup_a\Vert u_{\mu}^au_{\lambda}\Vert_{L_{t, x}^2}^2 \geq (\lambda^2 - \mu^2)\Vert u_{\mu}u_{\lambda}\Vert_{L_{t, x}^2}^2 \gtrsim \lambda^2\Vert u_{\mu}u_{\lambda}\Vert_{L_{t, x}^2}^2,
        \end{equation}
        and the desired conclusion is obtained from \eqref{17} and \eqref{10}.
    \end{proof}
\end{corollary}

Lastly, we consider the case of mBO equation \eqref{mBO}.

\begin{lemma}\label{lma4.5}
    Let $u$ satisfy $\partial_tu + \mathcal{H}\partial_x^2u = 0$. Then for any dyadic numbers $\mu$ and $\lambda$, it holds that 
    \begin{equation}\label{E}
        \begin{cases}
            \displaystyle\partial_t\left(\frac{u_{\mu}^2}{2}\right) + \partial_x\left(u_{\mu}D_xu_{\mu} - \int_{-\infty}^x\partial_yu_{\mu}D_yu_{\mu} \,{\rm d}y\right) = 0, \\ 
            \displaystyle\partial_t(u_{\lambda}D_xu_{\lambda}) + \partial_x\left(\frac{(D_xu_{\lambda})^2}{2}-u_{\lambda}\partial_x^2u_{\lambda} + \frac{(\partial_xu_{\lambda})^2}{2}\right) = 0.
        \end{cases}
    \end{equation}
    \begin{proof}
        First of all, it is easy to verify that $u_{\mu}$ and $u_{\lambda}$ satisfy the following equations:
        \begin{numcases}{}
            \label{20}\partial_tu_{\mu} + \partial_xD_xu_{\mu} = 0, \\ 
            \label{21}\partial_tu_{\lambda} + \partial_xD_xu_{\lambda} = 0.
        \end{numcases}
        On the one hand, multiplying by $u_{\mu}$ on both sides of \eqref{20}, we have  
        \begin{equation*}
            \partial_t\left(\frac{u_{\mu}^2}{2}\right) + \partial_xp = 0,
        \end{equation*}
        where 
        \begin{equation*}
            p = \int_{-\infty}^xu_{\mu}\partial_yD_yu_{\mu} \,{\rm d}y.
        \end{equation*}
        Integrating by parts and we obtain 
        \begin{equation*}
            p = u_{\mu}D_xu_{\mu} - \int_{-\infty}^x\partial_yu_{\mu}D_yu_{\mu} \,{\rm d}y.
        \end{equation*}
        Hence $u_{\mu}$ satisfies 
        \begin{equation*}
            \partial_t\left(\frac{u_{\mu}^2}{2}\right) + \partial_x\left(u_{\mu}D_xu_{\mu} - \int_{-\infty}^x\partial_yu_{\mu}D_yu_{\mu} \,{\rm d}y\right) = 0.
        \end{equation*}
        On the other hand, it follows from \eqref{21} that 
        \begin{numcases}{}
            \label{22}D_xu_{\lambda}\partial_tu_{\lambda} + \partial_x\left(\frac{(D_xu_{\lambda})^2}{2}\right) = 0, \\ 
            \label{23}u_{\lambda}\partial_tD_xu_{\lambda} + \partial_x\widetilde{p} = 0,
        \end{numcases}
        where 
        \begin{equation*}
            \widetilde{p} = \int_{-\infty}^xu_{\lambda}\partial_yD_y^2u_{\lambda}.
        \end{equation*}
        Note that $D_y^2 = -\partial_y^2$. Then we have 
        \begin{equation}
            \label{24}\widetilde{p} = -\int_{-\infty}^xu_{\lambda}\partial_y^3u_{\lambda} \,{\rm d}y = -u_{\lambda}\partial_x^2u_{\lambda} + \int_{-\infty}^x\partial_yu_{\lambda}\partial_y^2u_{\lambda} \,{\rm d}y = -u_{\lambda}\partial_x^2u_{\lambda} + \frac{(\partial_xu_{\lambda})^2}{2}.
        \end{equation}
        Combining \eqref{22}-\eqref{24} and we obtain 
        \begin{equation*}
            \partial_t(u_{\lambda}D_xu_{\lambda}) + \partial_x\left(\frac{(D_xu_{\lambda})^2}{2}-u_{\lambda}\partial_x^2u_{\lambda} + \frac{(\partial_xu_{\lambda})^2}{2}\right) = 0.
        \end{equation*}
        This completes the proof.
    \end{proof}
\end{lemma}

\begin{lemma}\label{lma2}
    For $a \in \mathbb{R}$, define 
    \begin{equation*}
        q_a(x) = \int_{-\infty}^x\partial_yu_{\mu}^aD_yu_{\mu}^a \,{\rm d}y.
    \end{equation*}
    Then there holds 
    \begin{equation*}
        \sup_a\int_0^T\int_{\mathbb{R}}u_{\lambda}D_xu_{\lambda} \cdot q_a(x) \gtrsim -\mu\lambda\sup_a\Vert u_{\mu}^au_{\lambda}\Vert_{L_{t, x}^2}^2.
    \end{equation*}
    \begin{proof}
        First of all, we have 
        \begin{equation*}
            q_a(x) = \frac{1}{2}\int_{-\infty}^x(\partial_yu_{\mu}^aD_yu_{\mu}^a + D_yu_{\mu}^a\partial_yu_{\mu}^a) \,{\rm d}y.
        \end{equation*}
        It follows that (ignoring the constant 1/2)
        \begin{equation*}
            \begin{aligned}
                \widehat{q_a}(\xi) = \frac{\widehat{\partial_xq_a}(\xi)}{\ii\xi} &= \frac{1}{\ii\xi}(\partial_xu_{\mu}^aD_xu_{\mu}^a + D_xu_{\mu}^a\partial_xu_{\mu}^a)^{\wedge}(\xi) \\
                &= \int_{\mathbb{R}}\left(\frac{\eta|\xi - \eta| + |\eta|(\xi - \eta)}{\xi}\right)\widehat{u_{\mu}^a}(\eta)\widehat{u^a_{\mu}}(\xi - \eta) \,{\rm d}\eta \\ 
                &= \int_{\mathbb{R}}\left(\frac{{\rm sgn}(\eta)|\xi - \eta| + (\xi - \eta)}{\xi}\right)|\eta|\widehat{u_{\mu}^a}(\eta)\widehat{u_{\mu}^a}(\xi - \eta) \,{\rm d}\eta.
            \end{aligned}
        \end{equation*}
        Let $\chi$ be the compactly supported smooth function defined in the proof of Lemma \ref{lma}.
        Using similar strategy and $\widehat{q}$ can be reformulated as 
        \begin{equation*}
        \begin{aligned}
            \widehat{q_a}(\xi) &= \mu\int_{\mathbb{R}}\left(\frac{{\rm sgn}(\eta)|\xi - \eta| + (\xi - \eta)}{\xi}\right)\chi\left(\frac{|\xi - \eta|}{\mu}\right)\chi\left(\frac{|\eta|}{\mu}\right)\frac{|\eta|}{\mu}\widehat{u_{\mu}^a}(\eta)\widehat{u_{\mu}^a}(\xi - \eta) \,{\rm d}\eta \\ 
            &:= \mu\int_{\mathbb{R}}K(\xi - \eta, \eta)\widehat{u_{\mu}^a}(\eta)\widehat{u_{\mu}^a}(\xi - \eta) \,{\rm d}\eta.
        \end{aligned}
        \end{equation*}
        It follows that 
        \begin{equation*}
            \begin{aligned}
                q_a(x) &= \mu\int_{\mathbb{R}}\int_{\mathbb{R}}K(\xi - \eta, \eta)\widehat{u_{\mu}^a}(\eta)\widehat{u_{\mu}^a}(\xi - \eta)\e^{\ii x \cdot \xi} \,{\rm d}\eta{\rm d}\xi \\ 
                &= \mu\int_{\mathbb{R}}\int_{\mathbb{R}}K(\tau, \tau')\widehat{u_{\mu}^a}(\tau)\widehat{u_{\mu}^a}(\tau')\e^{\ii x \cdot (\tau + \tau')} \,{\rm d}\tau{\rm d}\tau' \\ 
                &\sim \mu\int_{\mathbb{R}}\int_{\mathbb{R}}\mathscr{F}_{\tau, \tau'}^{-1}K(w, w')u_{\mu}^a(x - w)u_{\mu}^a(x - w') \,{\rm d}w{\rm d}w'.
            \end{aligned}
        \end{equation*}
        Note that $\Vert \mathscr{F}_{\tau, \tau'}^{-1}K\Vert_{L^1} = C < +\infty$, where $C$ is a constant independent of $\mu$, then we can obtain 
        \begin{equation*}
            \sup_a\int_0^T\int_{\mathbb{R}}u_{\lambda}D_xu_{\lambda} \cdot q_a \gtrsim -\mu\lambda\sup_a\Vert u_{\mu}^au_{\lambda}\Vert_{L_{t, x}^2}^2
        \end{equation*}
        from Minkowski's inequality. This completes the proof.
    \end{proof}
\end{lemma}

\begin{theorem}[Bilinear estimate II]\label{be-BO}
    Let $u$ satisfy $\partial_tu + \mathcal{H}\partial_x^2u = 0$.
    Then for any dyadic numbers $\lambda$ and $\mu$ satisfying $\mu \ll \lambda$, it holds that
    \begin{equation}\label{BE-BO}
        \Vert u_{\mu}u_{\lambda}\Vert_{L_{t, x}^2} \lesssim \lambda^{-\frac{1}{2}}\Vert u_{\mu}(0)\Vert_{L_x^2}\Vert u_{\lambda}(0)\Vert_{L_x^2}.
    \end{equation}
    \begin{proof}
        Set 
        \begin{equation*}
            \begin{cases}
                \displaystyle f^{11} = u_{\lambda}D_xu_{\lambda}, \\ 
                \displaystyle f^{12} = \frac{(D_xu_{\lambda})^2}{2}-u_{\lambda}\partial_x^2u_{\lambda} + \frac{(\partial_xu_{\lambda})^2}{2}, \\ 
                \displaystyle f^{21} = \frac{u_{\mu}^2}{2}, \\ 
                \displaystyle -f^{22} = u_{\mu}D_xu_{\mu} - \int_{-\infty}^x\partial_yu_{\mu}D_yu_{\mu} \,{\rm d}y.
            \end{cases}
        \end{equation*}
        Apply div-curl lemma to equation \eqref{E}. First of all, direct calculations show that 
        \begin{equation*}
            \begin{aligned}
                &\int_0^T\int_{\mathbb{R}}f^{11}f^{22} \\
                = &-\int_0^T\int_{\mathbb{R}}u_{\lambda}D_xu_{\lambda} \cdot u_{\mu}D_xu_{\mu} + \int_0^T\int_{\mathbb{R}}u_{\lambda}D_xu_{\lambda} \cdot \left(\int_{-\infty}^x\partial_yu_{\mu}D_yu_{\mu} \,{\rm d}y\right), \\
                &\int_0^T\int_{\mathbb{R}}f^{12}f^{21} \\ 
                = &\int_0^T\int_{\mathbb{R}}\left(\frac{1}{4}u_{\mu}^2(D_xu_{\lambda})^2 - \frac{1}{2}u_{\mu}^2u_{\lambda}\partial_x^2u_{\lambda} + \frac{1}{4}u_{\mu}^2(\partial_xu_{\lambda})^2\right) \\ 
                = &\frac{1}{4}\Vert u_{\mu}D_xu_{\lambda}\Vert_{L_{t, x}^2}^2 + \frac{1}{4}\Vert u_{\mu}\partial_xu_{\lambda}\Vert_{L_{t, x}^2}^2 - \frac{1}{2}\int_0^T\int_{\mathbb{R}}u_{\mu}^2(\partial_x(u_{\lambda}\partial_xu_{\lambda}) - (\partial_xu_{\lambda})^2) \\ 
                = &\frac{1}{4}\Vert u_{\mu}D_xu_{\lambda}\Vert_{L_{t, x}^2}^2 + \frac{3}{4}\Vert u_{\mu}\partial_xu_{\lambda}\Vert_{L_{t, x}^2}^2 + \frac{1}{2}\int_0^T\int_{\mathbb{R}}u_{\mu}\partial_xu_{\mu}\cdot u_{\lambda}\partial_xu_{\lambda}.
            \end{aligned}
        \end{equation*}
        Now substitute $u_{\mu}^a$ for $u_{\mu}$ in the aformentioned proof and keeping the rest unchanged. We have  
        \begin{align*}
            \int_0^T\int_{\mathbb{R}}f^{11}f^{22} + f^{12}f^{21} &= \frac{1}{4}\Vert u_{\mu}^aD_xu_{\lambda}\Vert_{L_{t, x}^2}^2 + \frac{3}{4}\Vert u_{\mu}^a\partial_xu_{\lambda}\Vert_{L_{t, x}^2}^2 + \frac{1}{2}\int_0^T\int_{\mathbb{R}}u_{\mu}^a\partial_xu_{\mu}^a\cdot u_{\lambda}\partial_xu_{\lambda} \\ 
            &- \int_0^T\int_{\mathbb{R}}u_{\mu}^aD_xu_{\mu}^a \cdot u_{\lambda}D_xu_{\lambda} + \int_0^T\int_{\mathbb{R}}u_{\lambda}D_xu_{\lambda} \cdot \left(\int_{-\infty}^x\partial_yu_{\mu}^aD_yu_{\mu}^a \,{\rm d}y\right).
        \end{align*}
        Similar to the proof Theorem \ref{be-KdV}, it follows from Lemma \ref{lma}, \ref{lma2} and their proofs that 
        \begin{gather}
            \label{25}\sup_a\Vert u_{\mu}^aD_xu_{\lambda}\Vert_{L_{t, x}^2}^2 \gtrsim \lambda^2\sup_a\Vert u_{\mu}^au_{\lambda}\Vert_{L_{t, x}^2}^2, \\
            \label{26}\sup_a\Vert u_{\mu}^a\partial_xu_{\lambda}\Vert_{L_{t, x}^2}^2 \gtrsim \lambda^2\sup_a\Vert u_{\mu}^au_{\lambda}\Vert_{L_{t, x}^2}^2, \\ 
            \label{27}\sup_a\int_0^T\int_{\mathbb{R}}u_{\mu}^a\partial_xu_{\mu}^a\cdot u_{\lambda}\partial_xu_{\lambda} \gtrsim -\mu\lambda\sup_a\Vert u_{\mu}^au_{\lambda}\Vert_{L_{t, x}^2}^2, \\
            \label{28}\sup_a\Vert u_{\mu}^au_{\lambda}\Vert_{L_{t, x}^2}^2 \gtrsim (\mu\lambda)^{-1}\sup_a\int_0^T\int_{\mathbb{R}}u_{\mu}^aD_xu_{\mu}^a\cdot u_{\lambda}D_xu_{\lambda}, \\ 
            \label{29}\sup_a\int_0^Tu_{\lambda}D_xu_{\lambda} \cdot \left(\int_{-\infty}^x\partial_yu_{\mu}^aD_yu_{\mu}^a {\rm d}y\right) \gtrsim -\mu\lambda\sup_a\Vert u_{\mu}^au_{\lambda}\Vert_{L_{t, x}^2}^2.
        \end{gather}
        Combining \eqref{25}-\eqref{29} and we have 
        \begin{equation}\label{30}
            \begin{aligned}
                \sup_a\int_0^T\int_{\mathbb{R}}f^{11}f^{22} - f^{21}f^{22} &\gtrsim \lambda^2\sup_a\Vert u_{\mu}^au_{\lambda}\Vert_{L_{t, x}^2}^2 - \mu\lambda\sup_a\Vert u_{\mu}^au_{\lambda}\Vert_{L_{t, x}^2}^2 \\
                &\gtrsim \lambda^2\sup_a\Vert u_{\mu}^au_{\lambda}\Vert_{L_{t, x}^2}^2 \\ 
                &\gtrsim \lambda^2\Vert u_{\mu}u_{\lambda}\Vert_{L_{t, x}^2}^2.
            \end{aligned}
        \end{equation}
        Lastly, since
        \begin{gather}
            \label{31}\Vert f^{11}(t)\Vert_{L^1_x} = \Vert u_{\lambda}(t)D_xu_{\lambda}(t)\Vert_{L^1_x} \leq \lambda\Vert u_{\lambda}(t)\Vert_{L_x^2}^2 = \lambda\Vert u_{\lambda}(0)\Vert_{L^2_x}^2, \quad \forall t, \\
            \label{32}\frac{{\rm d}}{{\rm d}t}\Vert f^{21}(t)\Vert_{L_x^1} = 0, \quad \Vert f^{21}(0)\Vert_{L^1_x} = \Vert (u_{\mu}^a)^2(0)\Vert_{L^1_x} = \Vert u_{\mu}(0)\Vert_{L^2_x}^2,
        \end{gather}
        then \eqref{BE-BO} follows from \eqref{30}-\eqref{32}.
    \end{proof}
\end{theorem}

\section{Proof of the Local Well-Posedness}\label{S5}

In this section, we first prove some estimates that we will use in the proof of Proposition \ref{prop4.1}, then we complete the proof of Proposition \ref{prop4.1}.
In the remaining part of this paper, we denote by $u := u^{(0)}$, where $u^{(0)}$ was defined in \eqref{31A}, and $r := \Vert \phi\Vert_{H^{\frac{1}{4}}_x}$ for simplicity.

\begin{lemma}
    Given dyadic number $\lambda$. Then there holds
    \begin{equation}\label{A}
        \Vert D_x^{\frac{3}{8}}u_{\lambda}\Vert_{L_{t, x}^4} \lesssim T^{\frac{1}{8}}\Vert D_x^{\frac{1}{4}}\phi_{\lambda}\Vert_{L^2_x}
    \end{equation}
    and equivalently, 
    \begin{equation}\label{ST1}
        \lambda^{\frac{3}{2}}\Vert u_{\lambda}\Vert_{L_{t, x}^4}^4 \lesssim T^{\frac{1}{2}}\lambda\Vert \phi_{\lambda}\Vert_{L^2_x}^4.
    \end{equation}
    Moreover, we have 
    \begin{gather} 
        \label{B}\sum_{\lambda \geq 1}\lambda^{\frac{3}{2}}\Vert u_{\lambda}\Vert_{L_{t, x}^4}^4 \lesssim T^{\frac{1}{2}}r^4, \\ 
        \label{C}\sum_{\lambda \geq 1}\lambda\Vert u_{\lambda}\Vert_{L_{t, x}^4}^{\frac{8}{3}} \lesssim T^{\frac{1}{3}}r^{\frac{8}{3}}.
    \end{gather}
    \begin{proof}
        We first notice that \eqref{ST1} can be derived by \eqref{A} and Bernstein's inequality.
        For \eqref{A}, by H\"older's inequality and Strichartz estimate we have 
        \begin{equation*}
            \Vert D_x^{\frac{3}{8}}u_{\lambda}\Vert_{L_{t, x}^4} \leq T^{\frac{1}{8}}\Vert D_x^{\frac{3}{8}}u_{\lambda}\Vert_{L_t^8L_x^4} \lesssim T^{\frac{1}{8}}\Vert D_x^{\frac{1}{4}}\phi_{\lambda}\Vert_{L^2_x},
        \end{equation*}
        which proves \eqref{A}, and it follows from \eqref{ST1} and the inclusion $\ell^2 \hookrightarrow \ell^4$ that
        \begin{align*}
            \sum_{\lambda \geq 1}\lambda^{\frac{3}{2}}\Vert u_{\lambda}\Vert_{L_{t, x}^4}^4 \lesssim T^{\frac{1}{2}}\sum_{\lambda \geq 1}\lambda\Vert \phi_{\lambda}\Vert_{L_x^2}^4 &= T^{\frac{1}{2}}\left(\sum_{\lambda \geq 1}\left(\lambda^{\frac{1}{4}}\Vert \phi_{\lambda}\Vert_{L_x^2}\right)^4\right)^{\frac{1}{4} \cdot 4} \\ 
            &\lesssim T^{\frac{1}{2}}\left(\sum_{\lambda \geq 1}\left(\lambda^{\frac{1}{4}}\Vert \phi_{\lambda}\Vert_{L_x^2}\right)^2\right)^{\frac{1}{2} \cdot 4} \\ 
            &= T^{\frac{1}{2}}r^4.
        \end{align*}
        Finally, the estimate \eqref{C} follows similarly due to the inclusion $\ell^2 \hookrightarrow \ell^{\frac{8}{3}}$.
    \end{proof}
\end{lemma}

\begin{lemma}\label{lma4.3}
    There holds 
    \begin{equation*}
        \sum_{1 \leq \mu, \lambda \lesssim \sigma}\sigma^{\frac{1}{6}}\Vert u_{\mu}\Vert_{L_{t, x}^4}^{\frac{4}{3}}\Vert \phi_{\lambda}\Vert_{L_x^2}^{\frac{4}{3}}\Vert \phi_{\sigma}\Vert_{L_x^2}^{\frac{4}{3}} \lesssim T^{\frac{1}{6}}r^{\frac{14}{3}}.
    \end{equation*}
    \begin{proof}
        We first prove 
        \begin{equation}\label{42}
            I := \sum_{1 \leq \mu, \lambda \lesssim \sigma}\sigma^{\frac{1}{6}}\Vert u_{\mu}\Vert_{L_{t, x}^4}^{\frac{4}{3}}\Vert \phi_{\lambda}\Vert_{L_x^2}^{\frac{4}{3}}\Vert \phi_{\sigma}\Vert_{L_x^2}^{\frac{4}{3}} \lesssim T^{\frac{1}{6}}r^{\frac{8}{3}}\sum_{\sigma \geq 1}\sigma^{\frac{1}{6} + \varepsilon}\Vert \phi_{\sigma}\Vert_{L_x^2}^{\frac{4}{3}},
        \end{equation}
        where $\varepsilon \in (0, 1/6)$ is arbitrary small. 
        
        To prove \eqref{42}, we divdie $I$ into three parts:
        \begin{align*}
            I &= \left(\sum_{1 \leq \mu \ll \lambda \lesssim \sigma} + \sum_{1 \leq \lambda \ll \mu \lesssim \sigma} + \sum_{1 \leq \mu \sim \lambda \lesssim \sigma}\right)\sigma^{\frac{1}{6}}\Vert u_{\mu}\Vert_{L_{t, x}^4}^{\frac{4}{3}}\Vert \phi_{\lambda}\Vert_{L_x^2}^{\frac{4}{3}}\Vert \phi_{\sigma}\Vert_{L_x^2}^{\frac{4}{3}} \\
            &=: I_1 + I_2 + I_3.
        \end{align*}
        It suffices to estimate $I_1$ and $I_3$, since the estimate of $I_2$ is similar to $I_1$. For $I_1$, by H\"older's inequality and \eqref{B} we have 
        \begin{equation}\label{43}
        \begin{aligned}
            I_1 &\lesssim \sum_{\sigma}\sigma^{\frac{1}{6}}\Vert \phi_{\sigma}\Vert_{L_x^2}^{\frac{4}{3}}\sum_{\lambda \lesssim \sigma}\Vert \phi_{\lambda}\Vert_{L_x^2}^{\frac{4}{3}}\sum_{1 \leq \mu \ll \lambda}\Vert u_{\mu}\Vert_{L_{t, x}^4}^{\frac{4}{3}} \\ 
            &\lesssim \sum_{\sigma}\sigma^{\frac{1}{6}}\Vert \phi_{\sigma}\Vert_{L_x^2}^{\frac{4}{3}}\sum_{\lambda \lesssim \sigma}\Vert \phi_{\lambda}\Vert_{L_x^2}^{\frac{4}{3}}\sum_{1 \leq \mu \ll \lambda}\mu^{\frac{1}{2}}\Vert u_{\mu}\Vert_{L_{t, x}^4}^{\frac{4}{3}} \\ 
            &\leq \sum_{\sigma}\sigma^{\frac{1}{6}}\Vert \phi_{\sigma}\Vert_{L_x^2}^{\frac{4}{3}}\sum_{\lambda \lesssim \sigma}\Vert \phi_{\lambda}\Vert_{L_x^2}^{\frac{4}{3}}\left(\left(\sum_{1 \leq \mu \ll \lambda}1^{\frac{3}{2}}\right)^{\frac{2}{3}}\left(\sum_{\mu \geq 1}\mu^{\frac{3}{2}}\Vert u_{\mu}\Vert_{L_{t, x}^4}^4\right)^{\frac{1}{3}}\right) \\ 
            &\lesssim T^{\frac{1}{6}}r^{\frac{4}{3}}\sum_{\sigma \geq 1}\sigma^{\frac{1}{6}}\Vert \phi_{\sigma}\Vert_{L_x^2}^{\frac{4}{3}}\sum_{1 \leq \lambda \lesssim \sigma}(\ln \lambda)^{\frac{2}{3}}\Vert \phi_{\lambda}\Vert_{L_x^2}^{\frac{4}{3}} \\ 
            &\lesssim T^{\frac{1}{6}}r^{\frac{4}{3}}\sum_{\sigma \geq 1}\sigma^{\frac{1}{6}}\Vert \phi_{\sigma}\Vert_{L_x^2}^{\frac{4}{3}}\sum_{1 \leq \lambda \lesssim \sigma}\lambda^{\frac{1}{3}}\Vert \phi_{\lambda}\Vert_{L_x^2}^{\frac{4}{3}} \\ 
            &\leq T^{\frac{1}{6}}r^{\frac{4}{3}}\sum_{\sigma \geq 1}\sigma^{\frac{1}{6}}\Vert \phi_{\sigma}\Vert_{L_x^2}^{\frac{4}{3}}\left(\left(\sum_{1 \leq \lambda \lesssim \sigma}1^3\right)^{\frac{1}{3}}\left(\sum_{\lambda \geq 1}\lambda^{\frac{1}{2}}\Vert \phi_{\lambda}\Vert_{L_x^2}^2\right)^{\frac{2}{3}}\right) \\ 
            &\lesssim T^{\frac{1}{6}}r^{\frac{4}{3}} \cdot r^{\frac{4}{3}} \sum_{\sigma \geq 1}\sigma^{\frac{1}{6}}(\ln\sigma)^{\frac{1}{3}}\Vert \phi_{\sigma}\Vert_{L_x^2}^{\frac{4}{3}}\\ 
            &= T^{\frac{1}{6}}r^{\frac{8}{3}}\sum_{\sigma \geq 1}\sigma^{\frac{1}{6} + \varepsilon}\Vert \phi_{\sigma}\Vert_{L_x^2}^{\frac{4}{3}},
        \end{aligned}
    \end{equation}
        where $\varepsilon$ can be arbitrary small.

        For $I_3$, it follows from H\"older's inequality and \eqref{B} that 
        \begin{equation}\label{44}
        \begin{aligned}
            I_3 &\lesssim \sum_{\sigma}\sigma^{\frac{1}{6}}\Vert \phi_{\sigma}\Vert_{L_x^2}^{\frac{4}{3}}\sum_{1 \leq \lambda \lesssim \sigma}\Vert u_{\sigma}\Vert_{L_{t, x}^4}^{\frac{4}{3}}\Vert \phi_{\sigma}\Vert_{L_x^2}^{\frac{4}{3}} \\ 
            &\lesssim \left(\sum_{\sigma \geq 1}\sigma^{\frac{1}{6}}\Vert \phi_{\sigma}\Vert_{L_x^2}^{\frac{4}{3}}\right)\left(\sum_{1 \leq \lambda \lesssim \sigma}\lambda^{\frac{1}{2}}\Vert u_{\lambda}\Vert_{L_{t, x}^4}^{\frac{4}{3}}\lambda^{\frac{1}{3}}\Vert \phi_{\lambda}\Vert_{L_x^2}^{\frac{4}{3}}\right) \\ 
            &\lesssim \left(\sum_{\sigma \geq 1}\sigma^{\frac{1}{6}}\Vert \phi_{\sigma}\Vert_{L_x^2}^{\frac{4}{3}}\right)\left(\sum_{\lambda \geq 1}\lambda^{\frac{3}{2}}\Vert u_{\lambda}\Vert_{L_{t, x}^4}^4\right)^{\frac{1}{3}}\left(\sum_{\lambda \geq 1}\lambda^{\frac{1}{2}}\Vert \phi_{\lambda}\Vert_{L_x^2}^2\right)^{\frac{2}{3}} \\ 
            &\lesssim T^{\frac{1}{6}}r^{\frac{8}{3}}\sum_{\sigma}\sigma^{\frac{1}{6}}\Vert \phi_{\sigma}\Vert_{L_x^2}^{\frac{4}{3}}.
        \end{aligned}
    \end{equation}
    Combining \eqref{43} with \eqref{44} and we complete the proof of \eqref{42}.

    Finally, by \eqref{42} we have 
    \begin{align*}
        I \lesssim \left(\sum_{\sigma \geq 1, \sigma^{\frac{1}{3} - \varepsilon}\Vert \phi_{\sigma}\Vert_{L_x^2}^{\frac{2}{3}} \leq 1} + \sum_{\sigma \geq 1, \sigma^{\frac{1}{3} - \varepsilon}\Vert \phi_{\sigma}\Vert_{L_x^2}^{\frac{2}{3}} > 1}\right)\sigma^{\frac{1}{6} + \varepsilon}\Vert \phi_{\sigma}\Vert_{L_x^2}^{\frac{4}{3}} =: I_1' + I_2'.
    \end{align*}
    On the one hand, for $I_1'$, we have 
    \begin{equation}\label{abcd}
        I_1' \leq \sum_{\sigma \geq 1}\sigma^{\frac{1}{6} + \varepsilon}\sigma^{2\varepsilon - \frac{2}{3}} = \sum_{\sigma \geq 1}\sigma^{3\varepsilon - \frac{1}{2}} = C < +\infty,
    \end{equation}
    where $C$ is a constant independent of $\varepsilon$; on the other hand, 
    \begin{align}\label{cd}
        I_2' \leq \sum_{\sigma \geq 1}\sigma^{\frac{1}{6} + \varepsilon}\Vert \phi_{\sigma}\Vert_{L_x^2}^{\frac{4}{3}} \cdot \sigma^{\frac{1}{3} - \varepsilon}\Vert \phi_{\sigma}\Vert_{L_x^2}^{\frac{2}{3}} = \sum_{\sigma \geq 1}\sigma^{\frac{1}{2}}\Vert \phi_{\sigma}\Vert_{L_x^2}^2 \leq r^2.
    \end{align}
    Now the desired conclusion follows from \eqref{abcd} and \eqref{cd}.
    \end{proof}
\end{lemma}

\begin{proof}[Proof of Proposition \ref{prop4.1}]
        We first prove \eqref{37}, \eqref{38} and \eqref{BE1}. Argue by induction, when $j = 0$, the estimate \eqref{37} and \eqref{BE1} follow from Strichartz estimate \eqref{SE-mKdV} and Corollary \ref{be-KdV}, respectively.
        For \eqref{38}, by H\"older's inequality, Bernstein's inequality and the inclusion $\ell^\frac{4}{3} \hookrightarrow \ell^2$ we have 
        \begin{align*}
            \Vert \partial_x(u^3)\Vert_{N(T)} &\lesssim T^{\frac{1}{8}}\left(\left\Vert\Vert S_1(\partial_x(u^3))\Vert_{L_x^{\frac{4}{3}}} + \left(\sum_{\lambda \geq 1}\left(\lambda^{\frac{1}{8}}\Vert P_{\lambda}(\partial_x(u^3))\Vert_{L_x^{\frac{4}{3}}}\right)^2\right)^{\frac{1}{2}}\right\Vert_{L_t^{\frac{4}{3}}}\right) \\ 
            &\lesssim T^{\frac{1}{8}}\left(\Vert S_1(\partial_x(u^3))\Vert_{L_{t, x}^{\frac{4}{3}}} + \left\Vert\left(\sum_{\lambda \geq 1}\left(\lambda^{\frac{9}{8}}\Vert P_{\lambda}(u^3)\Vert_{L_x^{\frac{4}{3}}}\right)^\frac{4}{3}\right)^{\frac{3}{4}}\right\Vert_{L_t^{\frac{4}{3}}}\right) \\ 
            &\lesssim T^{\frac{1}{8}}\left(\Vert S_1(\partial_x(u^3))\Vert_{L_{t, x}^{\frac{4}{3}}} + \left(\sum_{\lambda \geq 1}\left(\lambda^{\frac{9}{8}}\Vert P_{\lambda}(u^3)\Vert_{L_{x, t}^{\frac{4}{3}}}\right)^\frac{4}{3}\right)^{\frac{3}{4}}\right) \\ 
            &=: T^{\frac{1}{8}}(I_1 + I_2^{\frac{3}{4}}).
        \end{align*}
        We first estimate $I_2$:
        \begin{align*}
            I_2 \lesssim \left(\sum_{1 \leq \lambda \lesssim \lambda_1 \lesssim \lambda_2 \sim \lambda_3} + \sum_{1 \leq \lambda_1 \lesssim \lambda \lesssim \lambda_2 \sim \lambda_3} + \sum_{1 \leq \lambda_1, \lambda_2 \lesssim \lambda_3 \sim \lambda}\right)\lambda^{\frac{3}{2}}\Vert u_{\lambda_1}u_{\lambda_2}u_{\lambda_3}\Vert_{L_{x, t}^{\frac{4}{3}}}^{\frac{4}{3}} :=  I_{2a} + I_{2b} + I_{2c},
        \end{align*}
        For $I_{2a}$, by H\"older's inequality, \eqref{B} and \eqref{C} we have 
        \begin{align*}
            I_{2a} \lesssim \sum_{1 \leq \lambda \lesssim \mu \lesssim \sigma}\lambda^{\frac{3}{2}}\Vert u_{\mu}\Vert_{L_{t, x}^4}^{\frac{4}{3}}\Vert u_{\sigma}\Vert_{L_{t, x}^4}^{\frac{8}{3}} &\lesssim \sum_{\sigma}\Vert u_{\sigma}\Vert_{L_{t, x}^4}^{\frac{8}{3}}\sum_{\mu \lesssim \sigma}\Vert u_{\mu}\Vert_{L_{t, x}^4}^{\frac{4}{3}} \sum_{1 \leq \lambda \lesssim \mu}\lambda^{\frac{3}{2}} \\
            &\lesssim \sum_{\sigma}\Vert u_{\sigma}\Vert_{L_{t, x}^4}^{\frac{8}{3}}\sum_{1 \leq \mu \lesssim \sigma}\mu^{\frac{3}{2}}\Vert u_{\mu}\Vert_{L_{t, x}^4}^{\frac{4}{3}} \\
            &\lesssim \sum_{\sigma}\Vert u_{\sigma}\Vert_{L_{t, x}^4}^{\frac{8}{3}}\left(\left(\sum_{1 \leq \mu \lesssim \sigma}\mu^{\frac{3}{2}}\right)^{\frac{2}{3}}\left(\sum_{\mu \geq 1}\mu^{\frac{3}{2}}\Vert u_{\mu}\Vert_{L_{t, x}^4}^4\right)^{\frac{1}{3}}\right) \\
            &\lesssim T^{\frac{1}{6}}r^{\frac{4}{3}}\sum_{\sigma \geq 1}\sigma\Vert u_{\sigma}\Vert_{L_{t, x}^4}^{\frac{8}{3}} \\
            &\lesssim T^{\frac{1}{6}}r^{\frac{4}{3}} \cdot T^{\frac{1}{3}}r^{\frac{8}{3}} \\
            &= T^{\frac{1}{2}}r^4.
        \end{align*}
    For $I_{2b}$, using H\"older's inequality and bilinear estimate \eqref{BE-KdV} and we obtain 
    \begin{align*}
        I_{2b} \lesssim \sum_{1 \leq \mu \lesssim \lambda \lesssim \sigma}\lambda^{\frac{3}{2}}\Vert u_{\sigma}\Vert_{L_{t, x}^4}^{\frac{4}{3}}\Vert u_{\mu}u_{\sigma}\Vert_{L_{t, x}^2}^{\frac{4}{3}} \lesssim \sum_{1 \leq \mu \lesssim \lambda \lesssim \sigma}\lambda^{\frac{3}{2}}\sigma^{-\frac{4}{3}}\Vert u_{\sigma}\Vert_{L_{t, x}^4}^{\frac{4}{3}}\Vert \phi_{\mu}\Vert_{L_x^2}^{\frac{4}{3}}\Vert \phi_{\sigma}\Vert_{L_x^2}^{\frac{4}{3}} \lesssim \sum_{1 \leq \mu \lesssim \lambda \lesssim \sigma}\lambda^{\frac{1}{6}}\Vert u_{\sigma}\Vert_{L_{t, x}^4}^{\frac{4}{3}}\Vert \phi_{\mu}\Vert_{L_x^2}^{\frac{4}{3}}\Vert \phi_{\sigma}\Vert_{L_x^2}^{\frac{4}{3}}.
    \end{align*}
    Then it follows from H\"older's inequality and \eqref{B} that 
    \begin{align*}
        I_{2b} &\lesssim \sum_{\sigma}\Vert \phi_{\sigma}\Vert_{L_x^2}^{\frac{4}{3}}\Vert u_{\sigma}\Vert_{L_{t, x}^4}^{\frac{4}{3}}\sum_{\lambda \lesssim \sigma}\lambda^{\frac{1}{6}}\sum_{1 \leq \mu \lesssim \lambda}\Vert \phi_{\mu}\Vert_{L_x^2}^{\frac{4}{3}} \\ 
        &\lesssim \sum_{\sigma \geq 1}\Vert \phi_{\sigma}\Vert_{L_x^2}^{\frac{4}{3}}\Vert u_{\sigma}\Vert_{L_{t, x}^4}^{\frac{4}{3}}\sum_{\lambda \lesssim \sigma}\lambda^{\frac{1}{6}}\left(\left(\sum_{1 \leq \mu \lesssim \lambda}1^3\right)^{\frac{1}{3}}\left(\sum_{\mu \geq 1}\left(\mu^{\frac{1}{3}}\Vert \phi_{\mu}\Vert_{L_x^2}^{\frac{4}{3}}\right)^{\frac{3}{2}}\right)^{\frac{2}{3}} \right) \\ 
        &\lesssim r^{\frac{4}{3}}\sum_{\sigma \geq 1}\Vert \phi_{\sigma}\Vert_{L_x^2}^{\frac{4}{3}}\Vert u_{\sigma}\Vert_{L_{t, x}^4}^{\frac{4}{3}}\sum_{1 \leq \lambda \lesssim \sigma}\lambda^{\frac{1}{6} + \frac{1}{3}} \\ 
        &\lesssim r^{\frac{4}{3}}\sum_{\sigma \geq 1}\sigma^{\frac{1}{2}}\Vert \phi_{\sigma}\Vert_{L_x^2}^{\frac{4}{3}}\Vert u_{\sigma}\Vert_{L_{t, x}^4}^{\frac{4}{3}} \\ 
        &\lesssim r^{\frac{4}{3}}\sum_{\sigma \geq 1}\sigma^{\frac{1}{3}}\Vert \phi_{\sigma}\Vert_{L_x^2}^{\frac{4}{3}}\sigma^{\frac{1}{2}}\Vert u_{\sigma}\Vert_{L_{t, x}^4}^{\frac{4}{3}} \\ 
        &\lesssim r^{\frac{4}{3}}\left(\sum_{\sigma \geq 1}\sigma^{\frac{1}{2}}\Vert \phi_{\sigma}\Vert_{L_x^2}^2\right)^{\frac{2}{3}}\left(\sum_{\sigma \geq 1}\sigma^{\frac{3}{2}}\Vert u_{\sigma}\Vert_{L_{t, x}^4}^4\right)^{\frac{1}{3}} \\ 
        &\lesssim r^{\frac{4}{3}} \cdot r^{\frac{4}{3}} \cdot \left(T^{\frac{1}{2}}r^4\right)^{\frac{1}{3}} \\ 
        &= T^{\frac{1}{6}}r^4.
    \end{align*}
    Finally, for $I_{2c}$, by H\"older's inequality, bilinear estimate \eqref{BE-KdV} and Lemma \ref{lma4.3}, we have
    \begin{align*}
        I_{2c} \lesssim \sum_{1 \leq \mu, \sigma \lesssim \lambda}\lambda^{\frac{3}{2}}\Vert u_{\mu}\Vert_{L_{t, x}^4}^{\frac{4}{3}}\Vert u_{\sigma}u_{\lambda}\Vert_{L_{t, x}^2}^{\frac{4}{3}} &\lesssim \sum_{1 \leq \mu, \sigma \lesssim \lambda}\lambda^{\frac{1}{6}}\Vert u_{\mu}\Vert_{L_{t, x}^4}^{\frac{4}{3}}\Vert \phi_{\sigma}\Vert_{L_x^2}^{\frac{4}{3}}\Vert \phi_{\lambda}\Vert_{L_x^2}^{\frac{4}{3}} \lesssim T^{\frac{1}{6}}r^{\frac{14}{3}}.
    \end{align*}
    To conlude,
    \begin{equation*}
        I_2 \lesssim T^{\frac{1}{2}}r^4 + T^{\frac{1}{6}}r^4 + T^{\frac{1}{6}}r^{\frac{14}{3}}.
    \end{equation*}
    For $I_1$, we have  
    \begin{align*}
        I_1 \lesssim \left\Vert\sum_{\lambda < 1}\lambda P_{\lambda}(u^3)\right\Vert_{L_{t, x}^{\frac{4}{3}}} \lesssim \left\Vert\sum_{\lambda_1, \lambda_2, \lambda_3}u_{\lambda_1}u_{\lambda_3}u_{\lambda_3}\right\Vert_{L_{t, x}^{\frac{4}{3}}} &\leq \left\Vert\left(\sum_{\lambda}u_{\lambda}\right)^3\right\Vert_{L_{t, x}^{\frac{4}{3}}} \\
        &= \left\Vert\sum_{\lambda}u_{\lambda}\right\Vert_{L_{t, x}^4}^3  \\
        &\lesssim \Vert u\Vert_{S(T)}^3 \\
        &\lesssim \left(T^{\frac{1}{8}}r\right)^3 \\
        &= T^{\frac{3}{8}}r^3.
    \end{align*}
    It follows that
    \begin{align*}
        \Vert \partial_x(u^3)\Vert_{N(T)} \leq CT^{\frac{1}{8}}\left(T^{\frac{1}{8}}r^3 + T^{\frac{3}{8}}r^3 + T^{\frac{1}{8}}r^{\frac{7}{2}}\right) = C\left(T^{\frac{1}{4}}r^3 + T^{\frac{1}{2}}r^3 + T^{\frac{1}{4}}r^{\frac{7}{2}}\right) \leq r,
    \end{align*}
    where we choose $T > 0$ such that $CT^{\frac{1}{4}}r^{\frac{5}{2}} \ll 1$. This completes the proof of \eqref{38} when $j = 1$. 
    Moreover, by \eqref{SE-mKdV} we have 
        \begin{equation}\label{39A}
            \Vert u^{(1)}\Vert_{S(T)} \leq C\left(\Vert \phi\Vert_{H^{\frac{1}{4}}_x} + \Vert \partial_x(u^3)\Vert_{N(T)}\right).
        \end{equation} 
    Therefore, we have that 
    \begin{equation*}
        \Vert u^{(1)}\Vert_{S(T)} \leq 2Cr,
    \end{equation*}

    Suppose now that \eqref{37}, \eqref{38} and \eqref{BE1} hold for $j = k - 1 \in \mathbb{N}$. On the one hand, by Strichartz estimate \eqref{SE-mKdV} we have 
    \begin{equation}\label{st3}
        \Vert u^{(k)}\Vert_{S(T)} \leq 2Cr.
    \end{equation}
    On the other hand, it is straightforward to verify that $u_{\mu}^{(k)}$ and $u_{\lambda}^{(k)}$ satisfy the following equations:
    \begin{equation*}
        \begin{cases}
            \displaystyle \partial_t\left(\frac{(u^{(k)}_{\mu})^2}{2}\right) + \partial_x\left(u^{(k)}_{\mu}\partial_x^2u^{(k)}_{\mu} - \frac{(\partial_xu^{(k)}_{\mu})^2}{2}\right) = -u^{(k)}_{\mu}P_{\mu}\partial_x((u^{(k - 1)})^3), \\ 
            \displaystyle \partial_t\left(\frac{(u^{(k)}_{\lambda})^2}{2}\right) + \partial_x\left(u^{(k)}_{\lambda}\partial_x^2u^{(k)}_{\lambda} - \frac{(\partial_xu^{(k)}_{\lambda})^2}{2}\right) = -u^{(k)}_{\lambda}P_{\lambda}\partial_x((u^{(k - 1)})^3).
        \end{cases}
    \end{equation*} 
    Apply div-curl lemma \eqref{d-c} to the equations above. Since
    \begin{gather*}
        \Vert u^{(k)}_{\mu}P_{\mu}(\partial_x((u^{(k - 1)})^3))\Vert_{L_{t, x}^1} \leq \Vert u^{(k)}_{\mu}\Vert_{L_t^8L_x^4}\Vert P_{\mu}(\partial_x((u^{(k - 1)})^3))\Vert_{L^{\frac{8}{7}}_tL_x^{\frac{4}{3}}}, \\ 
        \Vert u^{(k)}_{\lambda}P_{\lambda}(\partial_x((u^{(k - 1)})^3))\Vert_{L_{t, x}^1} \leq \Vert u^{(k)}_{\lambda}\Vert_{L_t^8L_x^4}\Vert P_{\lambda}(\partial_x((u^{(k - 1)})^3))\Vert_{L^{\frac{8}{7}}_tL_x^{\frac{4}{3}}},
    \end{gather*}
    it follows from \eqref{st3} and the inductive hypothesis that the following bilinear estimate is also valid when $\mu \ll \lambda$: 
    \begin{equation*}
        \Vert u_{\mu}^{(k)}u_{\lambda}^{(k)}\Vert_{L_{t, x}^2} \leq 3C'C\lambda^{-1}\Vert u_{\mu}^{(k)}(0)\Vert_{L_x^2}\Vert u_{\lambda}^{(k)}(0)\Vert_{L_x^2} = 3CC'\lambda^{-1}\Vert \phi_{\mu}\Vert_{L_x^2}\Vert \phi_{\lambda}\Vert_{L_x^2},
    \end{equation*}
    that is, the bilinear estimate \eqref{BE1} holds true when $j = k$. Finally, repeating the proof of the case $j = 0$ and we can also know that \eqref{38} hold for $j = k$ (where we choose $T > 0$ sufficiently small such that $3CC'T^{\frac{1}{4}}r^{\frac{5}{2}} \leq 1/2$) and by Strichartz estimate, we have 
    \begin{equation*}
        \Vert u^{(k + 1)}\Vert_{S(T)} \leq 2Cr.
    \end{equation*}
    Thus \eqref{37}, \eqref{38} and \eqref{BE1} hold for all $j$ by induction.

    Finally, for \eqref{39} and \eqref{40A}, using div-curl lemma, Strichartz estimate \eqref{SE-mKdV} and induction proof, it can be shown similarly that the following holds for $j = 1, 2, \cdots$:
    \begin{gather*}
        \Vert u^{(j + 1)} - u^{(j)}\Vert_{S(T)} \leq C\Vert \partial_x((u^{(j)})^3 - (u^{(j - 1)})^3)\Vert_{N(T)} \leq \frac{1}{2}\Vert u^{(j)} - u^{(j - 1)}\Vert_{S(T)}, \\ 
        \Vert \partial_x((u^{(j + 1)})^3 - (u^{(j)})^3)\Vert_{N(T)} \leq C'\Vert u^{(j + 1)} - u^{(j)}\Vert_{S(T)} \leq \frac{1}{2}\Vert \partial_x((u^{(j)})^3 - (u^{(j - 1)})^3)\Vert_{N(T)}.
    \end{gather*}
    Here we choose $T > 0$ sufficiently small. This completes the proof.
\end{proof}

\begin{remark}
    It follows from the inclusion \eqref{ab} that our method is also applicable to the case $s \geq 1/4$ (resp., $s \geq 1/2$) as stated in Theorem \ref{thm-mKdV} (resp., Theorem \ref{thm-mBO}). 
    Futhermore, the LWP for IVP \eqref{mgBO} when $s \geq 1/2 - \alpha/4$ can be also obtained in a similar way.
\end{remark}

\textbf{Acknowledgement.} L. Tu would like to thank Prof. Baoping Liu and Prof. Ning-An Lai for their helpful comments.

This work was supported by the National Natural Science Foundation of China [No. 12171097], the Key Laboratory of Mathematics
for Nonlinear Sciences (Fudan University), the Ministry of Education of China and Shanghai Key Laboratory for Contemporary Applied Mathematics.

\bibliographystyle{plain}
\bibliography{refs}

\end{document}